%% file: main.tex
\documentclass[11pt]{article}
\usepackage[utf8]{inputenc}
\usepackage{
    algorithm, 
    amsmath,
    amsfonts, 
    amsthm, 
    amssymb,
    authblk,
    booktabs,
    color,
    dcolumn,
    enumerate,
    fullpage,
    hyperref,
    mathrsfs,
    multirow,
    tabularx,
    tikz,
}
\usepackage[noend]{algpseudocode}
\usepackage[round]{natbib}

\newtheorem{theorem}{Theorem}[]

\newtheorem{lemma}{Lemma}[]
\newtheorem{corollary}{Corollary}[]

\newtheorem{observation}{Observation}[]

\newtheorem{definition}{Definition}
\newtheorem{conjecture}{Conjecture}[]

\newcommand{\N}{\mathbb{N}}
\renewcommand{\emptyset}{\varnothing}
\newcommand{\cutvx}[1]{\ensuremath{V_\mathrm{cut}}}
\newcommand{\cutes}{\ensuremath{E_\mathrm{cut}}}
\newcommand{\eant}{\ensuremath{E_\mathrm{ant}}}
\newcommand{\bct}[1]{\ensuremath\mathrm{BCT}(#1)}

\newcommand{\refp}[1]{(\ref{#1})}

\DeclareMathOperator{\opp}{antip}

\usetikzlibrary{decorations.pathreplacing}

\hypersetup{
    colorlinks=true,
    linkcolor=blue,
    citecolor=blue,
}

\title{A Polynomial Time Algorithm for Computing the Strong Rainbow Connection Numbers of Odd Cacti}
\author{{\Large Logan A. Smith, David T. Mildebrath, and Illya V. Hicks}}
\date{\today}
\affil{
    \textit{Department of Computational and Applied Mathematics}\\
    \textit{Rice University}\\
    \textit{Houston, TX 77005}}

\begin{document}

\maketitle

\begin{abstract}\noindent
We consider the problem of computing the strong rainbow connection number $src(G)$ for cactus graphs $G$ in which all cycles have odd length. We present a formula to calculate $src(G)$ for such odd cacti which can be evaluated in linear time, as well as an algorithm for computing the corresponding optimal strong rainbow edge coloring, with polynomial worst case run time complexity. Although computing $src(G)$ is NP-hard in general, previous work has demonstrated that it may be computed in polynomial time for certain classes of graphs, including cycles, trees and block clique graphs. This work extends the class of graphs for which $src(G)$ may be computed in polynomial time.
\end{abstract}

\section{Introduction}

Let $G$ be a non-empty simple connected graph, and let $c:E(G)\to\{1,\dots,k\}$ for $k\in\N$ be a $k$-coloring of the edges of $G$ (note that $c$ is not necessarily proper, so that adjacent edges may be the same color). The graph $G$ is (strongly) rainbow connected with respect to $c$ if, for every pair of vertices $u,v\in{V(G)}$, there exists a (shortest) $u,v$ path $P$ in $G$ such that no two edges in $E(P)$ are the same color. The (strong) rainbow connection number $rc(G)$ ($src(G)$) is the minimum number of colors $k$ for which there exists a (strong) rainbow $k$-coloring of $G$. In general, we have that $rc(G)\leq{src(G)}$.

The concept of (strong) rainbow connection was first introduced by \citet{chartrand2008}, and was originally intended to model the flow of classified information between government agencies in the aftermath of the terrorist attacks of September 11, 2001 \citep{chartrand2009}. It has since been applied in other areas, including the routing of information over secure computer networks (i.e.~``onion routing'' \citep{reed1998}). In addition to these applications, rainbow connection is of theoretical interest, and has recently garnered significant attention (see \citet{li2013a} for a review).

In this work, we focus on computing the strong rainbow connection number $src(G)$ for \textit{odd cactus} graphs---that is, cactus graphs which do not contain a cycle of even length (a graph $G$ is a cactus if every edge in $E(G)$ is contained in at most one cycle in $G$).
We emphasize that this is distinct from the closely related rainbow connectivity problem which, given an edge coloring $c$, asks whether $c$ strongly rainbow connects $G$. For general graphs $G$, determining whether $src(G)\leq{k}$ is NP-hard for $k\geq3$, even when $G$ is bipartite \citep{ananth2011} (the same is true of $rc(G)$ \citep{chakraborty2009}). However, for certain classes of graphs, $src(G)$ may be computed in polynomial time. These include: trees, cycles, wheels, complete multipartite graphs \citep{chartrand2008}, fan and sun graphs \citep{sy2013}, stellar graphs \citep{shulhany2016} and block clique graphs \citep{keranen2018}.
To our knowledge, no polynomial time algorithm is known for computing $src(G)$ (or indeed $rc(G)$) in cactus graphs. In this work, we provide a formula to compute $src(G)$ for odd cacti $G$ which can be evaluated in $O(n)$ time.



Cacti have previously been considered in the context of rainbow coloring. In particular, \citet{uchizawa2013} show that, given a fixed edge coloring $c$ of a cactus $G$, determining whether $c$ strongly rainbow connects $G$ can be done in polynomial time (although this problem is NP-complete on general graphs, including interval outerplanar and $k$-regular graphs \citep{lauri2016}). A particularly useful property of \textit{odd} cacti which is not shared by other cacti, is that they are \textit{geodetic}---i.e., every pair of vertices in the graph is connected by a unique shortest path \citep{stemple1968}.

Of the work studying (strong) rainbow connection in structured graphs, we highlight two particular results related to our own. The first is the work of
\citet{alvasamos2017}, who present a method for computing a variant of $rc(G)$ for directed cacti. In the class of graphs considered by the authors, each pair of vertices is connected by a unique (directed) path, and thus $src(G)=rc(G)$. This identity does not hold for the standard $src(G)$ introduced by \citet{chartrand2008}, which we consider here. More background on the directed variant of rainbow connection can be found in \citet{dorbec2014}.

The second work we highlight is that of \citet{keranen2018}, who present a linear time algorithm for computing $src(G)$ when $G$ is a block clique graph (i.e.~a graph whose blocks are all cliques). A block clique graph in which each block contains at most 3 vertices is also an odd cactus, and thus our results coincide with the results of \citet{keranen2018} for such graphs. However, block clique graphs are a special class of chordal graphs. Odd cacti are not chordal in general, and thus many of the techniques used by \citet{keranen2018} do not extend to the graphs we consider here.


Our main result is a simple formula for $src(G)$ when $G$ is an odd cactus. This formula relies on the notion of an antipodal vertex-edge pair.  We formalize this notion below, but intuitively, given an edge $u_1u_2$ in an odd cycle $C$, the antipodal vertex to $u_1u_2$ is the unique vertex $v:=\opp(u_1u_2)$ in $V(C)$ which is the same distance from both $u_1$ and $u_2$ (i.e.~on the opposite side of the cycle from $u_1u_2$).
With this concept in hand, we prove that for any odd cactus $G$, 
\begin{equation}\label{eqn:main-formula}\tag{M}
    src(G)=\begin{cases}
        \tfrac{1}{2}\big(m+|\cutes|+|\mathcal{S}_1|-|\eant|\big),&G\text{ is not an odd cycle},\\
        (n+1)/2,&G=C_n\text{ for }n\geq5,\\
        1&G=C_3.
        \end{cases}
\end{equation}
where $m$ is the number of edges contained in $G$; $n$ in the number of vetices contained in $G$; $\cutes$ is the set of cut edges in $G$; $\eant$ is the the set of edges in $E(G)$ whose antipodal vertices are cut vertices; and $|\mathcal{S}_1|$ is equal to the number of pairs of cut vertices $(u,v)$ such that (1) $u$ and $v$ are contained in the same cycle in $G$, and (2) the shortest path between $u$ and $v$ contains no other cut vertices, and no edges in $\eant$. 



We prove the correctness of \refp{eqn:main-formula} according to the following outline. First, we introduce the notion of a \textit{black white partition} of $G$, in which every vertex and edge in $G$ is assigned to be either ``black'' or ``white''. The key property of black white partitions is that for every pair of black edges in a valid black white partition of $G$, there exists a shortest path between a pair of vertices in $V(G)$ which traverses both black edges. Since odd cacti are geodetic, in any strong rainbow coloring of $G$ each black edge must assigned a distinct color, and thus the number of black edges in any valid black white partition provides a lower bound on $src(G)$ (\autoref{thm:bwcoloring}). We next introduce the concept of a \textit{segment} of a cycle in $G$, and use this idea to construct a particular black white partition of $G$ (\autoref{thm:bw}). Using this particular black white partition, we introduce a polynomial time algorithm which produces a strong rainbow coloring of $G$ (\autoref{thm:src}). The coloring produced by our algorithm satisfies the lower bound of \autoref{thm:bwcoloring} with equality, and is thus optimal. The formula \refp{eqn:main-formula} follows as a corollary from these three primary results (\autoref{cor:formula}).


\section{Notation and Preliminaries}

Let $G$ be a graph with vertex set $V(G)$ and edge set $E(G)$. A subgraph of $G$ is called a {\em component} of $G$ if it is a maximal connected subgraph of $G$. If a graph has a single component it is called connected. For the remainder of the paper, we assume that all graphs are simple, connected, and non-empty. A vertex $v$ is called a {\em cut vertex} if $G - v$ has strictly more distinct components than $G$. The set of cut vertices of $G$ is denoted as $\cutvx{G}$. A maximal connected subgraph $B$ of $G$ such that no vertex can be removed from $B$ to disconnect it is called a {\em block} of $G$, and the set of blocks within $G$ is denoted as $\mathcal{B}(G)$. A graph known as the {\em block-cut tree} $\bct{G}$ of $G$ can be constructed such that $V(\bct{G}) = \cutvx{G} \cup \mathcal{B}(G)$ and two vertices in $\bct{G}$ are adjacent if and only if one vertex corresponds to a block in $G$ and the other corresponds to a cut vertex in $G$ contained in that block. 
A graph $G$ is called a {\em cactus} graph, or simply a cactus, if each edge in $E(G)$ is contained in the edge set of at most one cycle contained in $G$. Equivalently, a graph is a cactus if every block in $G$ is either a cycle or the graph consisting of two vertices joined by an edge. Moreover, we call a cactus graph $G$ {\em odd} if $G$ contains no cycles of even length. Note that if $G$ is an odd cactus, then for any pair of vertices $u,v$ in $V(G)$ there is a unique shortest $u,v$ path contained in $G$. 

\begin{lemma}[\citet{chartrand2008}]
$src(C_3) = 1$ and for any odd $n \geq 5$, $src(C_n) = \frac{n+1}{2}$.
\end{lemma}

Additionally, we note that while odd length cycles are odd cacti, a graph can be identified as an odd length cycle in worst case linear time complexity and the strong rainbow connection numbers of cycles are well known. Since these cases can be easily dealt with, we will often consider odd cacti which are distinct from cycles. \newline

\begin{observation}[\citet{li2013b}] \label{lem:cut-edges}
Let $G$ be a graph with cut edges $e_1, e_2$. Then for any strong rainbow coloring $c$ of $E(G)$,  $c(e_1) \neq c(e_2)$.
\end{observation} 

\begin{definition}
Let $C$ be an odd cycle, and $v_1, v_2, v_3$ be vertices in $V(C)$. If $d(v_1, v_2) = d(v_1, v_3)$ and $v_2 v_3 \in E(C)$, then vertex $v_1$ is antipodal to edge $v_2v_3$. Additionally, define $\opp(v_2v_3, C) = v_1$ and $\opp(v_1, C) = v_2 v_3$.
\end{definition}

We note that in the case of an odd cactus $G$, any edge $e\in{E(G)}$ can be contained in the edge set of at most one cycle $C$ contained in $G$. Thus when considering odd cacti, $\opp(e, C)$ can be denoted as $\opp(e)$ without ambiguity. Additionally, in the case that an edge $e$ is not contained in the edge set of any odd cycle contained in $G$, let $\opp(e) = \emptyset$. Similarly, if $C$ is a cycle contained in $G$ and $v \not \in V(C)$, then let $\opp(v, C) = \emptyset$. 

\section{Cycles and Cycle Segments in Odd Cacti}

Let $G$ be a  cactus. Since formulae for $src(G)$ are known in the case that $G$ is either cycle or a tree, assume that $G$ is not a cycle, but does contain some cycle $C$. Additionally, as $G$ is assumed to be connected and each block in a cactus is either a cycle or a pair of vertices sharing an edge, some vertex $v$ in $V(C)$ must also be contained in the vertex set of a block $B$ in $G$, distinct from $C$. Then vertex $v$ is contained in the vertex sets of two or more blocks in $G$ so $v$ must be a cut vertex in $G$. As surmised in the following observation, each cycle in $G$ must then contain a cut vertex. 

\begin{observation}\label{obs:cut-boyes}
Let $G$ be a cactus that is not a cycle. For any cycle $C$ contained in $G$, at least one vertex in $V(C)$ is a cut vertex in $G$.  
\end{observation}


Now, let $\eant=\{e\in{E(G)}:\opp(e)\text{ is a cut vertex in }G\}$. Because all cactus graphs are outerplanar, we may fix a plane embedding of $G$---for the remainder of the paper, the notion of moving ``clockwise'' and ``counterclockwise'' in a cycle is defined relative to this embedding.
Let $C$ be a cycle in $G$ and let $v_1\in{V(C)}$ be a cut vertex in $G$ (the existence of $v_1$ follows from \autoref{obs:cut-boyes}).

For the remainder of the paper, we will be interested in various alternating sequences of vertices and edges. These sequences may begin with vertices \textit{or} edges. For simplicity, we refer to such sequences generically using the notation $x_ix_{i+1}\dots{x_k}$, where the $x_i$ may be either a vertex or an edge. We emphasize that these sequences will always alternate, so that if $x_i$ is a vertex, then $x_{i+1}$ is an edge (and vice versa).

Let the sequence $W=x_1x_2\dots{x_{2n(C)+1}}$ be the closed trail in cycle $C$ beginning and ending at vertex $v_1$ moving clockwise (note that $x_1=x_{2n(C)+1}=v_1$). A \textit{cycle segment} of the cycle $C$ is a subsequence of $W$ of the form $S=x_jx_{j+1}\dots{x_{j+k}}$ for $j\in\{2,\dots,2n(C)-k\}$ such that $x_{j-1},x_{j+k+1}\in\cutvx{G}\cup\eant$ and $x_{j+\ell}\not\in\cutvx{G}\cup\eant$ for all $\ell\in\{0,\dots,k\}$. Given a cycle segment $S$, let $V(S)$ denote the set of vertices contained in $S$, and let $E(S)$ denote the set of edges contained in $S$.

An example of a set of cycle segments is shown in \autoref{fig:segments}. Note that the cycle segments in a cycle $C$ are the same regardless of the choice of initial cut vertex $v_1$.

\begin{observation}
Let $G$ be a cactus that is not a cycle. 
\begin{enumerate}
    \item If $C$ is a cycle in $G$, then $V(C)$ is partitioned by the vertex sets of the cycle segments of $C$ and $\cutvx{G} \cap V(C)$. Furthermore, $E(C)$ is partitioned by the edge sets of the cycle segments of $C$ and $\eant\cap E(C)$.
    \item $V(G)$ is partitioned by the vertex sets of cycle segments in $G$, $\cutvx{G}$, and the set of leaves in $G$. $E(G)$ is partitioned by the edge sets of cycle segments in $G$, $\eant$ and $\cutes$. 
\end{enumerate}
\end{observation}

We categorize a cycle segment $S=x_j\dots{x_{j+k}}$ based on whether the ``boundary elements'' $x_{j-1}$ and $x_{j+k+1}$ are edges or vertices. Specifically, $S$ must lie in exactly one of the following sets:
$\mathcal{S}_1=\{S:x_{j-1},x_{j+k+1}\in\cutvx{}\}$,
$\mathcal{S}_2=\{S:x_{j-1}\in\cutvx{},x_{j+k+1}\in\eant\}$,
$\mathcal{S}_3=\{S:x_{j-1}\in\eant,x_{j+k+1}\in\cutvx{}\}$, and
$\mathcal{S}_4=\{S:x_{j-1},x_{j+k+1}\in\eant\}$.

\begin{figure}
\centering
\input{segments.tex}
\caption{Example of cycle segments. Dark vertices are in $\cutvx{}$, and dark edges are in $\eant$. The vertices in boxes are cut vertices, and the edges in boxes are antipodal to cut vertices.}
\label{fig:segments}
\end{figure}


\begin{lemma}\label{lem:opp-seg}
Let $C$ be a cycle in $G$, and let $W=x_1\dots{x_{2n(C)+1}}$ be a closed trail in $C$ such that $x_1=x_{2n(C)+1}\in\cutvx{G}$.
If $S=x_jx_{j+1}\dots{x_{j+k}}$ is a cycle segment in a cycle $C$, then the subsequence $S'=\opp(x_j,C)\opp(x_{j+1},C)\dots\opp(x_{j+k},C)$ is a cycle segment in $C$. Moreover, if $S\in\mathcal{S}_1$ (resp.~$\mathcal{S}_2$, $\mathcal{S}_3$, $\mathcal{S}_4$), then $S'\in\mathcal{S}_4$ (resp.~$\mathcal{S}_3$, $\mathcal{S}_2$, $\mathcal{S}_1$).
\end{lemma}
\begin{proof}
Observe that $z$ immediately succeeds $y$ in $W$ if and only if $\opp(z,C)$ immediately succeeds $\opp(y,C)$ in $W$.
Hence, to show that $S'$ is a cycle segment, it suffices to show that $\opp(x_{j-1},C),\opp(x_{j+k+1},C)\in\cutvx{G}\cup\eant$, and $\opp(x_{j+\ell},C)\not\in\cutvx{G}\cup\eant$ for all $\ell\in\{0,\dots,k\}$. By definition of $\eant$, we have for any $v\in{V(C)}$ that $v\in\cutvx{G}\Leftrightarrow\opp(v,C)\in\eant$, and similarly for any edge $e\in{E(C)}$, $e\in\eant\Leftrightarrow\opp(e,C)\in\cutvx{G}$. Generically, given $x\in{E(C)\cup{V(C)}}$, $x\in\cutvx{G}\cup\eant\Leftrightarrow\opp(x,C)\in\cutvx{G}\cup\eant$.

Suppose for contradiction that $\opp(x_{j-1},C)\not\in\cutvx{G}\cup\eant$. Then $x_{j-1}\not\in\cutvx{G}\cup\eant$, contradicting the fact that $S$ is a cycle segment. A symmetric argument shows that $\opp(x_{j+k+1},C)\in\cutvx{G}\cup\eant$. Finally, suppose for contradiction that $\opp(x_{j+\ell},C)\in\cutvx{G}\cup\eant$ for some $\ell\in\{0,\dots,k\}$. Then $x_{j+\ell}\in\cutvx{G}\cup\eant$, contradicting the fact that $S$ is a cycle segment. We conclude that $S'$ is a cycle segment.

Now suppose that $S\in\mathcal{S}_1$, so that $x_{j-1},x_{j+k+1}\in\cutvx{}$. Then $\opp(x_{j-1},C),\opp(x_{j+k+1},C)\in\eant$. Because $x_{j-1}$ immediately succeeds $x_j$ in $W$, we have that $\opp(x_{j-1},C)$ immediately succeeds $\opp(x_j,C)$ in $W$. Similarly, because $x_{j+k}$ immediately succeeds $x_{j+k+1}$ in $W$, $\opp(x_{j+k},C)$ immediately succeeds $\opp(x_{j+k+1},C)$ in $W$. We conclude that $S'\in\mathcal{S}_4$. A symmetric argument establishes the other cases ($S\in\mathcal{S}_2$, $\mathcal{S}_3$, $\mathcal{S}_4$).
\end{proof}

\begin{observation}\label{obs:opp-seg}
\autoref{lem:opp-seg} implies that the cycle segments in $\mathcal{S}_1$ and $\mathcal{S}_4$ come in pairs, as do the cycle segments in $\mathcal{S}_2$ and $\mathcal{S}_3$. More precisely, for every cycle segment $S\in\mathcal{S}_1$ in a cycle $C$, there exists a cycle segment $S'\in\mathcal{S}_4$ which is also contained in $C$. The same result holds for $\mathcal{S}_2$ and $\mathcal{S}_3$.
\end{observation}

\section{Black-White Partitions of Odd Cacti}

\noindent
We now introduce the notion of a \textit{black white partition} of an odd cactus $G$, and prove a series of results showing how black white partitions can be used to provide a lower bound for the strong rainbow connection numbers of these graphs. Additionally, we use the cycle segment taxonomy introduced in the previous section to outline a general construction for a special black white partition (\autoref{thm:bw}). In later sections, we see that \autoref{alg:1} provides an edge coloring that both strongly rainbow connects its input odd cactus and attains the lower bound that the black white partition constructed by \autoref{thm:bw} implies, thus proving the optimality of the strong rainbow coloring provided by \autoref{alg:1}.

\begin{definition}\label{def:bwc}
Let $G$ be an odd cactus, $(V_B, V_W)$ be a partition $V(G)$, and $(E_B, E_W)$ be a partition of $E(G)$. The tuple $(V_B, V_W, E_B, E_W)$ is called a black white partition of $G$ if each of the following properties hold:
\begin{enumerate}
    \item {For any cycle $C$ contained in $G$ and any edge $e \in E(C)$, $e \in E_B$ if and only if $\opp(e) \in V_W$ and $e \in E_W$ if and only if $\opp(e) \in V_B$}
    \item If $v_1 v_2 \in E_B$, then $\{v_1, v_2\} \subseteq V_B$
    \item If $v \in V_W$, then for any edge $uv \in E(G)$, $uv \in E_W$
    \item For every cut vertex $v \in V_W$, there exists a component $K$ in $G - v$ such that $E_B \subseteq E(K)$. 
\end{enumerate}
\end{definition}

\autoref{lem:going-around} will be used in conjunction with  \autoref{def:bwc} to prove \autoref{thm:bwcoloring}.

\begin{lemma}  \label{lem:going-around}
Let $C$ be an odd cycle, $u_1u_2 \in E(C)$, $w = \opp(u_1u_2)$. For any $v \in V(C) \backslash\{u_1, u_2, w\}$, if $d(v, u_2) < d(v, u_1)$ and $P$ is the shortest $v,u_2$ path, then the shortest $u_1, v$ path is $C[V(P) \cup \{u_1\}]$. 
\end{lemma}

\begin{proof}
Let $C$ be an odd cycle of length $n$, $u_1u_2 \in E(C)$, vertex $w = \opp(u_1 u_2)$. Let the vertices of $C$ be labeled $v_1, \dots, v_n$, such that $E(C) = \{v_i v_{i+1} : i \in \{1, \dots, n-1\}\} \cup \{v_1 v_n\}$. Without loss of generality, assume that $u_1 = v_n$ and $u_2 = v_1$. The vertex $w$ then must be $v_j$ where $j = \frac{n+1}{2}$. Consider any $v_k$ such that $k \not \in \{1, j, n\}$ and $d(v_1, v_k) < d(v_n, v_k)$. Consider the two $v_1, v_k$ paths contained in $C$, $P_1 = C[\{v_i : 1 \leq i \leq k\}]$ and $P_2 = C[\{v_i : k \leq i \leq n\} \cup \{v_1\}]$, and the two $v_n, v_k$ paths contained in $C, P_3 = C[\{v_i : k \leq i \leq n\}]$ and $P_4 = C[\{v_i : 1 \leq i \leq k\} \cup \{v_k\}]$. Note that if $k > j$, then since $j-1 = n-j$, it must be that $d(v_k, v_1) = \min(|E(P_1)|, |E(P_2)|) = \min(k-1, n-k+1) \geq \min(j, n-k+1) = \min(n-j+1, n-k+1) = n-k+1 > n-k \geq \min(k, n-k) = \min(|E(P_3)|, |E(P_4)|) = d(v_k, v_n)$, a contradiction. Finally, since it must be the case that $k < j$ so paths $P_2, P_3$ both have lengths strictly greater than paths $P_1, P_4$, respectively. Thus $P_1$ is the shortest $v_1, v_k$ path contained in $C$ and $P_4$ is the shortest $v_n, v_k$ path contained in $C$. By construction, $P_4 = C[V(P_1) \cup \{v_n\}]$. 
\end{proof}

\begin{theorem} \label{thm:bwcoloring}
Let $(V_B, V_W, E_B, E_W)$ be a black white partition of an odd cactus $G$. Then for any distinct pair of edges $u_1u_2, u_3u_4\in{E_B}$ and strong rainbow coloring $c$ of $E(G)$, $c(u_1u_2) \neq c(u_3u_4)$.
\end{theorem}

\begin{proof}
We prove the theorem directly by verifying that 
for any pair of edges $u_1u_2, u_3u_4 \in E_B$, and for any strong rainbow edge coloring $c$ of $E(G)$, it must be that $c(u_1 u_2) \neq c(u_3 u_4)$. In many cases, this is shown by proving that a pair of vertices can be chosen such that the edge set of the shortest path between these vertices contains both $u_1u_2$ and $u_3u_4$. Odd cacti are geodetic, thus every pair of vertices in an odd cactus have a unique shortest path connecting them. Any strong rainbow coloring $c$ must then map the edges in the edge set of each shortest path to distinct colors, thus the existence of a shortest path that traverses $u_1u_2$ and $u_3u_4$ implies that $c(u_1 u_2) \neq c(u_3 u_4)$. We also note that because $G$ is a cactus, each edge is contained in the edge set of at most one cycle. We continue the proof by considering five cases. The first case we consider is the case in which $\{u_1, u_2\} \cap \{u_3, u_4\} \neq \emptyset$. We assume in all subsequent cases that $\{u_1, u_2\} \cap \{u_3, u_4\} = \emptyset$. The remaining cases are as follows: neither $u_1u_2$ nor $u_3u_4$ are contained in any cycles in $G$, exactly one of $u_1u_2, u_3u_4$ are contained in a cycle, $u_1u_2, u_3u_4$ are contained in the same cycle, and $u_1u_2, u_3u_4$ are contained in different cycles.
 
\textit{\underline{Case 1:} Vertex sets $\{u_1, u_2\}$ and $\{u_3, u_4\}$ are not disjoint}. Since edges $u_1u_2$ and $u_3,u_4$ are assumed to be distinct, the intersection of these sets is a single vertex. Without loss of generality, assume that $u_2 = u_3$. If the edge $u_1 u_4 \in E(G)$, then $G[\{u_1, u_2, u_4\}]$ is a 
3-cycle in which the black vertex $u_4$ is antipodal to the black edge $u_1u_2$, a contradiction. Hence $u_1 u_4 \not \in E(G)$, and $d(u_1, u_4) \geq 2$. The path $G[\{ u_1, u_2, u_4 \}]$ has length two is thus the shortest $u_1, u_4$ path contained in $G$ and contains both edge $u_1u_2$ and edge $u_3u_4$.

\textit{\underline{Case 2:} Neither $u_1u_2$ nor $u_3u_4$ are contained in any cycles}. If neither $u_1u_2$ nor $u_3u_4$ are contained in any cycles in $G$, then since $G$ is a cactus $u_1u_2$ and $u_3u_4$ must both be cut edges. By \autoref{lem:cut-edges}, $c(u_1u_2) \neq c(u_3u_4)$.

\textit{\underline{Case 3:} Exactly one of $u_1u_2, u_3u_4$ are contained in a cycle}. Let $u_1 u_2, u_3 u_4$ be edges in $E_B$, $u_1 u_2$ be a cut edge, and $C$ be a cycle such that $u_3 u_4 \in E(C)$. Since $u_1 u_2$ is a cut edge, $G - u_1 u_2$ is disconnected and exactly one of $u_1, u_2$ are contained in the vertex set of the component which contains $u_3 u_4$. Assume without loss of generality that vertex $u_2$ is contained in vertex set of the component that contains $u_3 u_4$. Since there is no $u_1, u_3$ path in $G - u_1 u_2$, any $u_1, u_3$ path in $G$ must contain edge $u_1 u_2$. By symmetry, any $u_1, u_4$ path in $G$ must also contain edge $u_1 u_2$.

Let $B_1, B_2$ be the distinct blocks in $G$ such that $u_1 u_2 \in E(B_1)$ and $u_3 u_4 \in E(B_2)$. Since the block-cut tree $\bct{G}$ is a tree, there exists a unique $B_1, B_2$ path $P_B$ in $\bct{G}$. Let $s$ be the vertex in $P_B$ that is adjacent to $B_2$. Since $s$ is a cut vertex in $G$ which separates $B_1$ and $B_2$, every path from $u_1$ to $u_3$ or $u_4$ in $G$ must contain vertex $s$. Additionally, $u_1 u_2$ and $u_3 u_4$ are contained in the edge sets of different components of $G-s$, it must be the case that $s \in V_B$. Since $s \in V_B, u_3 u_4 \in E_B$, and $s \in V(B_2)=V(C)$, it follows that $s \neq \opp(u_3 u_4)$ and thus $d(s, u_3) \neq d(s, u_4)$. Assume without loss of generality that $d(s, u_3) < d(s, u_4)$. By \autoref{lem:going-around}, the shortest $s,u_4$ path $P$ in $C$ must contain $u_3 u_4$. Since every path between two vertices in the same block of a graph is contained in that block, $P$ is also the shortest $s, u_4$ path in $G$. 

Since each $u_1, u_4$ path in $G$ must contain vertex $s$, $d(u_1, u_4) = d(u_1, s) + d(s, u_4)$. Let $P'$ be the shortest $u_1, s$ path in $G$. Since $P'' = P \cup P'$ is a $u_1, u_4$ path in $G$ that has length $d(u_1, s) + d(s, u_4)$, $P''$ is the unique shortest $u_1, u_4$ path in $G$. Since $E(P'')$ contains $u_1 u_2$ and $u_3 u_4$ by construction, $c(u_1 u_2) \neq c(u_3 u_4)$. 

\textit{\underline{Case 4:} Both $u_1u_2$ and $u_3u_4$ are contained in a single cycle}. Let $C$ be the odd cycle in $G$ containing $u_1u_2, u_3u_4$. Since $u_3 u_4$ cannot be antipodal to $u_1$, we assume without loss of generality that $d(u_1, u_3) < d(u_1,u_4)$. Similarly, $\opp(u_1 u_2) \neq (u_3)$, so either $d(u_1, u_3) < d(u_2,u_3)$ or $d(u_1, u_3) > d(u_2,u_3)$. Again, without loss of generality, we assume that $d(u_2, u_3) < d(u_1,u_3)$.

By \autoref{lem:going-around}, since $u_3 \not \in \{u_1, u_2, \opp(u_1 u_2)\}$ and $d(u_2, u_3) < d(u_3, u_1)$, the shortest $u_1, u_3$ path $P' = C[V(P) \cup \{u_1\}]$ where $P$ is the shortest $u_2, u_3$ path in $C$. Additionally, since $u_1 \not \in \{ u_3, u_4, \opp(u_3 u_4)\}$ and $d(u_1, u_3) > d(u_1,u_4)$, again by \autoref{lem:going-around}, the shortest $u_1, u_4$ path in $C$ is $P'' = C[V(P') \cup \{u_4\}]$. Since $C$ is a block in $G$ and $u_1, u_4$ are vertices in $V(C)$, any $u_1, u_4$ path in $G$ is contained in $C$ and thus $P''$ is also the shortest $u_1, u_4$ path in $G$. Since $P''$ is an induced subgraph that contains vertices $u_1, u_2, u_3, u_4$, edge set $E(P'')$ contains both $u_1u_2$ and $u_3 u_4$.

\textit{\underline{Case 5:} Edges $u_1 u_2$ and $u_3 u_4$ are contained in distinct cycles}.
Let $C_1, C_2$ be cycles in $G$ such that $u_1 u_2 \in E(C_1)$ and $u_3 u_4 \in E(C_2)$. Let $P_B$ be the unique $C_1, C_2$ path in $\bct{G}$, and $s_1, s_2$ be vertices in $V(P_B)$ such that $s_1, s_2$ are adjacent to $C_1, C_2$, respectively. Note that $s_1$ is a cut vertex in $G$ and that if $s_1 \in V_W$ then the edge sets of multiple components in $G-v$ contain edges in $E_B$. This would violate property 4 of black white partitions, so it must be that $s_1 \in V_B$. Similarly, $s_2$ must also be in $V_B$. Since $u_1 u_2, u_3 u_4 \in E_B$, property 1 of black white partitions implies that both $\opp(u_1 u_2) \neq s_1$ and $\opp(u_3 u_4) \neq s_2$ must hold. Assume without loss of generality that $d(u_1, s_1) > d(u_2, s_1)$ and $d(u_4, s_2) > d(u_3, s_2)$. Since every $u_1, u_4$ path in $G$ must contain vertex $s_2$, it follows that $d(u_1, u_4) = d(u_1, s_2) + d(s_2, u_4)$. Similarly every $u_1, s_2$ path must contain $s_1$, so $d(u_1, s_2) = d(u_1, s_1) + d(s_1, s_2)$. Combining these, $d(u_1, u_4) = d(u_1, s_1) + d(s_1, s_2) + d(s_2, u_4)$. 

By \autoref{lem:going-around}, the shortest $u_1, s_1$ path $P_1$ contained in $C_1$ must contain edge $u_1 u_2$ and the shortest $u_4, s_2$ path $P_2$ contained in $C_2$ must contain edge $u_3 u_4$. Let $P_3$ be the shortest $s_1, s_2$ path in $G$. Then path $P = P_1 \cup P_3 \cup P_2$ is a $u_1, u_4$ path in $G$ with length $d(u_1, s_1) + d(s_1, s_2) + d(s_2, u_4) = d(u_1, u_4)$. Path $P$ is thus the unique shortest $u_1, u_4$ path in $G$ and $E(P)$ contains both $u_1 u_2$ and $u_3 u_4$.

\end{proof}

By \autoref{thm:bwcoloring}, if $(V_B, V_W, E_B, E_W)$ is a valid black white partition of an odd cactus $G$, then for any valid strong rainbow coloring $c$ of $G$, $c$ must map each pair of edges in $E_B$ to different colors. This provides a lower bound for the strong rainbow connection number of $G$:

\begin{corollary} \label{cor:bwbound}
For any odd cactus $G$, if $(V_B, V_W, E_B, E_W)$ is a black white partition of $G$ then $|E_B| \leq src(G)$.
\end{corollary}

As \autoref{cor:bwbound} shows that a graph's black white partitions provide lower bounds for the strong rainbow connection number, it is natural to ask which black white partition provides the best possible bound for arbitrary odd cacti---that is, which black white partition maximizes $|E_B|$. Such a partition is constructed by \autoref{thm:bw}.
In \autoref{sec:alg}, we will see that the black white partition constructed in \autoref{thm:bw} is optimal in the sense that $|E_B|=src(G)$, and thus that this partition indeed attains the maximum value of $|E_B|$ over all black white partitions of $G$.


\begin{theorem} \label{thm:bw}
For any odd cactus $G$ which is not a cycle, the sets
\begin{subequations}\label{eqn:bw-all}
\begin{align}
\label{eqn:bw-vb}
V_B&=\textstyle\big(\bigcup_{S\in\mathcal{S}_1\cup\mathcal{S}_2}V(S)\big)\cup\cutvx{}\cup\{v\in{V(G)}:v\text{ is a leaf}\},\\
\label{eqn:bw-vw}
V_W&=\textstyle\big(\bigcup_{S\in\mathcal{S}_3\cup\mathcal{S}_4}V(S)\big),\\
\label{eqn:bw-eb}
E_B&=\textstyle\big(\bigcup_{S\in\mathcal{S}_1\cup\mathcal{S}_2}E(S)\big)\cup\cutes,\\
\label{eqn:bw-ew}
E_W&=\textstyle\big(\bigcup_{S\in\mathcal{S}_3\cup\mathcal{S}_4}E(S)\big)\cup\eant,
\end{align}
form a black white partition of $G$.
\end{subequations}
\end{theorem}
\begin{proof}
Because $E(G)$ is partitioned by $\cutes$, $\eant$, and the edge sets of the cycle segments in $G$, it follows that $E_B$ and $E_W$ partition $E(G)$. Similarly, since $V(G)$ is partitioned by the leaves, cut vertices, and vertex sets of the cycle segments of $G$, it follows that $V_B$ and $V_W$ partition $V(G)$.

We now show that each of the properties 2-4 of black white partitions hold for $(V_B, V_W, E_B, E_W)$. For any edge $u_1 u_2$ in $E_B$, if $u_1 u_2$ is a cut edge then $u_1, u_2$ must each either be a leaf or a cut vertex and thus are both in $V_B$. If $u_1 u_2$ is not a cut edge then it must be in the edge set of a cycle segment $S$ in $\mathcal{S}_1$ or $\mathcal{S}_2$. If $u_1$ is a cut vertex then $u_1$ is in $V_B$ by construction, and if not then $u_1$ is in $V(S)$ and thus is in $V_B$. By symmetry $u_2$ must also be in $V_B$, showing that in all cases property 2 holds. Given that property 2 holds, its contrapositive statement ensures that if a vertex $v$ is in $V_W$, then no edge incident upon $v$ can be in $E_B$. Since $E(G)$ is partitioned by $E_B, E_W$, any edge incident upon $v$ must then be in $E_W$, thus property 3 also holds. Finally, since no cut vertex in $G$ is contained in the set $V_W$, property 4 holds vacuously.

We complete our proof 
by showing that property 1 holds. Consider an arbitrary edge $e$ in a cycle $C$ contained in $G$. If $e \in \eant$, then $e \in E_W$ and $\opp(e)$ is a cut vertex. All cut vertices in $G$ are in $V_B$, so $\opp(e) \in V_B$. If $e \not \in \eant$, then $e$ must be contained in a cycle segment $S$ contained in $C$. If $S$ is in $\mathcal{S}_1$ or $\mathcal{S}_2$, by \autoref{lem:opp-seg}, edge $e$ must be antipodal to a vertex in $V(S')$ where $S'$ is a segment of $C$ in $\mathcal{S}_3$ or $\mathcal{S}_4$, respectively (cf.~\autoref{obs:opp-seg}). In either case, $e \in E_B$ and since $V(S')\subseteq{V_W}$, $e$ is antipodal to a vertex in $V_W$. Similarly, If $S$ is in $\mathcal{S}_3$ or $\mathcal{S}_4$, edge $e$ must be antipodal to a vertex in $V(S')$ where $S'$ is a segment of $C$ in $\mathcal{S}_1$ or $\mathcal{S}_2$. In either case then, $e \in E_W$ and since $V(S') \subseteq V_B$, $e$ is antipodal to a vertex in $V_B$. Since each edge in an odd cactus is antipodal to at most one vertex and the vertex antipodal to each edge is in $V_B, V_W$ if $e$ is in $E_W, W_B$, respectively, no vertex in $V_B, V_W$ can be antipodal to an edge in $E_B, E_W$, respectively. Thus $(V_B, V_W, E_B, E_W)$ must satisfy property 1, and is therefore a valid black white partition of $G$. 
\end{proof}

\section{An Algorithm for Computing $src$ in Odd Cacti}\label{sec:alg}

We next introduce \autoref{alg:1}, which takes an odd cactus graph $G$ as its input and provides a strong rainbow coloring of $G$ as its output. The functioning of the algorithm requires the concept of a \textit{separation} of an odd cactus $G$ with respect to an antipodal pair $(e,v)$, which we define here.

\begin{definition}
Let $G$ be an odd cactus, $C$ be a cycle in $G$, and let $e\in{E(C)}$ such that $v:=\opp(e)\in\cutvx{G}$.
Let $K_1,\dots,K_m$ denote the components of $G-v$, labeled such that $e\in{E(K_1)}$.
The separation of $G$ with respect to the antipodal pair $(e,v)$ is the pair $(G_1,G_2)$ of connected subgraphs of $G$ such that $G_1=G[V(K_1)\cup\{v\}]$ and $G_2=G[\bigcup_{i=2}^mV(K_i)\cup\{v\}]$
\end{definition}

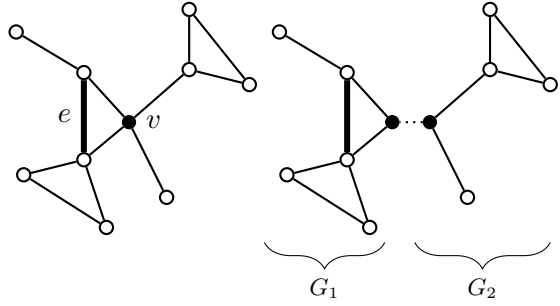
\begin{figure}\centering
\input{separation.tex}
\caption{Example separation of an odd cactus with respect to the antipodal pair $(e,v)$ highlighted in bold. Note that the separated graphs $G_1$ and $G_2$ share no edges and only one vertex ($v$).}
\label{fig:sep}
\end{figure}

Note that, given an antipodal pair $(e,v)$ such that $v\in\cutvx{G}$, the separation of $G$ with respect to $(e,v)$ is unique. Moreover, $(E(G_1),E(G_2))$ is a partition of $E(G)$, $V(G_1)\cap{V(G_2)}=\{v\}$, and $V(G_1)\cup{V(G_2)}=V(G)$. An example of a separation is illustrated in \autoref{fig:sep}.

\autoref{alg:1} outlines the procedure for computing an optimal strong rainbow coloring of an odd cactus $G$. The sets required for the algorithm, $\cutes$, $\cutvx{}$, $\eant$ and $\mathcal{S}_i$ for $i=1,2$ can be computed in worst-case $O(n)$ time complexity as follows. First, using the algorithm of \citet{tarjan1972}, the set $\cutvx{}$ and all the blocks of $G$ can be computed in $O(n)$ time. Because $G$ is a cactus, each block is either a cycle or a cut edge (in addition to its two adjacent vertices). Hence $\cutes$ can be identified in linear time. Finally, the sets $\mathcal{S}_1$, $\mathcal{S}_2$ and $\eant$ can be identified as follows: for each cycle $C$ (found previously), iterate over each antipodal vertex-edge pair $(e,v)$ to identify the edges in $\eant$. The sets $\mathcal{S}_1$ and $\mathcal{S}_2$ can be identified by performing one more iteration over each cycle.

Because the sets $\cutes$, $\eant$ and $\mathcal{S}_1$ can be identified in linear time, we conclude that the formula \refp{eqn:main-formula} can be evaluated in linear time, and thus $src(G)$ can be computed in linear time for odd cacti. Furthermore, the ``for each'' statements in lines \ref{line:5} and \ref{line:6}, and respectively lines \ref{line:11} and \ref{line:12}, iterate through each edge in each cycle segment in $\mathcal{S}_1 \cup \mathcal{S}_2$ while performing a constant number of operations in each step. Thus, lines \ref{line:1}-\ref{line:15} terminate in linear time. Finally, the ``for each'' loop in line \ref{line:16} iterates through each edge $e\in\eant$, and computes a separation of $G$ about $(e, \opp(e))$ which can be done in linear time through a process similar to component identification. Given this separation, selecting an edge as done in line \ref{line:18} can be done in linear time as well, thus the for each statement in line \ref{line:16} terminates in overall $O(n^2)$ time. The worst case time complexity of \autoref{alg:1} is thus bounded by $O(n^2)$. We note that it may be possible to modify \autoref{alg:1} so that the algorithm terminates in worst case linear time complexity---we do not consider this task here.

\begin{algorithm} 
\caption{Compute an optimal strong rainbow coloring of an odd cactus $G$.}
\algnewcommand{\GoTo}{\State\textbf{go to}~}%
\algnewcommand{\algorithmicforeach}{\textbf{for each}}
\algdef{S}[FOR]{ForEach}[1]{\algorithmicforeach\ #1\ }
\algdef{S}[FOR]{For}[1]{\algorithmicfor\ #1\ }
\algdef{S}[IF]{If}[1]{\algorithmicif\ #1\ }
\algdef{S}[IF]{ElsIf}[1]{\algorithmicelse\ \algorithmicif\ #1\ }
\algdef{S}[WHILE]{While}[1]{\algorithmicwhile\ #1\ }
\label{alg:1} 
\begin{algorithmic}[1]
\Require $G$ is a non-empty, odd cactus distinct from cycles.
\State color $\gets 0$\label{line:1}
\ForEach{edge $e\in\cutes$}\label{line:2}
    \State color $\gets$ color + 1\label{line:3}
    \State $c(e) \gets$ color\label{line:4}
\EndFor
\ForEach{segment $S=e_1v_1\dots{v_{\ell-1}e_\ell} \in \mathcal{S}_1$ of a cycle $C$ in $G$}\label{line:5}
    \ForEach{$i=1,\dots,\ell$}\label{line:6}
        \State color $\gets$ color + 1\label{line:7}
        \State $c(e_i) \gets$ color\label{line:8}
        \If{$i\neq\ell$}\label{line:9}
            \State $c(\opp(v_i,C))\gets$ color  \label{line:10} 
        \EndIf
    \EndFor
\EndFor
\ForEach{segment $S=e_1v_1\dots{e_{\ell}v_\ell} \in \mathcal{S}_2$ of a cycle $C$ in $G$}\label{line:11}
    \ForEach{$i=1,\dots,\ell$}\label{line:12}
        \State color $\gets$ color + 1\label{line:13}
        \State $c(e_i) \gets$ color\label{line:14}
        \State $c(\opp(v_i,C))\gets$ color   \label{line:15} 
    \EndFor
\EndFor
\ForEach{edge $e\in\eant$}\label{line:16}
    \State $(G_1,G_2)\gets$ separation of $G$ with respect to $(e,\opp(e))$\label{line:17}
    \State $e'\gets$ arbitrary edge in $E(G_2)$ that is in $\cutes$ or in a cycle segment in $\mathcal{S}_1\cup\mathcal{S}_2$\label{line:18}
    \State $c(e)\gets{c(e')}$\label{line:19}
\EndFor
\Return c 
\end{algorithmic}
\end{algorithm}

In order to prove the correctness of \autoref{alg:1} (\autoref{thm:src}), we require the following three lemmas (Lemmas \ref{lem:l0}--\ref{lem:opp-comp}).

\begin{lemma}\label{lem:l0}
Let $G$ be an odd cactus, let $a,b\in{V(G)}$, and let $P$ be an $a,b$ path. Let $C$ be a cycle contained in $G$, and let $e\in{E(C)}$. If $e\in{E(P)}$ and $\opp(e)\in{V(P)}$, then $P$ is not a shortest $a,b$ path.
\end{lemma}

\begin{proof}
Let $v_1,\dots,v_\ell$ be an ordering of $V(P)$ such that $v_iv_{i+1}\in{E(P)}$ for all $i=1,\dots,\ell-1$, where $a=v_1$, $b=v_\ell$ and $\ell=|V(P)|$. Let $v_{i_1}$ be the first vertex in the ordering such that $v_{i_1}\in{V(C)}$, and let $v_{i_2}$ be the last vertex in the ordering such that $v_{i_2}\in{V(C)}$. Because $G$ is a cactus, $v_j\in{V(C)}$ for all $j=i_1,\dots,i_2$. Next, let $w_1,\dots,w_r$ be an ordering of $V(C)$ such that $w_1=v_{i_1},w_2=v_{i_1+1},\dots,w_q=v_{i_2}$ and $w_jw_{j+1}\in{E(C)}$ for all $j=1,\dots,r-1$. Let $P_1$ be the $w_1,w_q$ path in $C$ composed of the edges $(w_1,w_2),(w_2,w_3),\dots,(w_{q-1},w_q)$, and let $P_2$ be the $w_1,w_q$ path in $C$ composed of the edges $(w_1,w_r),(w_r,w_{r-1}),\dots,(w_{q+1},w_q)$. Then there exist indices $k_1\in\{1,\dots,q-1\}$ and $k_2\in\{1,\dots,q\}$ such that $e=w_{k_1}w_{k_1+1}$ and $\opp(e)=w_{k_2}$. Without loss of generality, assume that $k_1<k_2$ (otherwise, we may relabel the vertices $a$ and $b$ and repeat the argument). We have that
$|E(P_1)| \geq d(w_1,w_{k_1})+d(w_{k_1},w_{k_1+1})+d(w_{k_1+1},w_{k_2})+d(w_{k_2},w_q){\geq}d(w_{k_1},w_{k_1+1})+d(w_{k_1+1},w_{k_2})=(r+1)/2$. Combining this with the fact that $r=|E(P_1)|+|E(P_2)|$, we have that $|E(P_2)|\leq(r-1)/2$, and thus $E(P_2)<E(P_1)$. Hence, the path composed of edges 
\[
    (a=v_1,v_2),\dots,(v_{i_1-1},v_{i_1}=w_1),(w_1,w_r),\dots,(w_{q+1},w_q=v_{i_2}),(v_{i_2},v_{i_2+1}),\dots,(v_{\ell-1},v_\ell=b)
\]
is an $a,b$ path in $G$ that is strictly shorter than $P$, completing the proof (see \autoref{fig:paths}).
\end{proof}

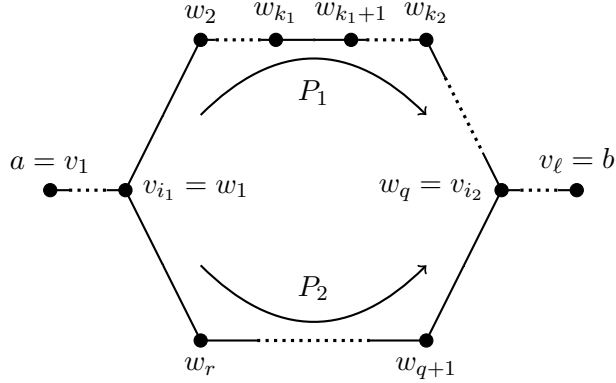
\begin{figure}\centering
\input{paths.tex}
\caption{Illustration of the proof of \autoref{lem:l0}.}
\label{fig:paths}
\end{figure}


\newcommand{\parti}[1]{(V_B^{#1}, V_W^{#1}, E_B^{#1}, E_W^{#1})}

\begin{lemma}\label{lem:bct}
Let $(V_B,V_W,E_B,E_W)$ be the black white partition of $G$ given by \refp{eqn:bw-all},
and let $B\in\mathcal{B}(G)$ be a block whose corresponding node in $\bct{G}$ is a leaf. Then $E(B)\cap{E_B}\neq\emptyset$.
\end{lemma}

\begin{proof}
Suppose that $B$ is a block corresponding to a leaf node in $\bct{G}$. If $B$ consists of two vertices connected by an edge in $\cutes$, then the claim holds because $\cutes\subseteq{E_B}$. 
Otherwise $B$ is a cycle. Let  $v_1e_1\dots{v_{(\ell-1)/2}}e_{(\ell-1)/2}v_{(\ell+1)/2}\dots{e_\ell}v_1$ be the closed trail contained in $B$, where $v_1$ is the unique cut vertex in $V(B)$ (uniqueness follows because $B$ is a leaf in $\bct{G}$) and $e_{(\ell-1)/2}=\opp(v_1,B)$. Clearly, $v_1e_1\dots{v_{(\ell-1)/2}}$ is a segment of $B$ in $\mathcal{S}_2$, and similarly $v_{(\ell+1)}\dots{e_\ell}v_1$ is a segment of $B$ in $\mathcal{S}_3$. It follows that $E(B)\cap{E_B}\neq\varnothing$ because, in particular, it contains edge $e_1$.
\end{proof}




\begin{lemma} \label{lem:opp-comp}
Let $G$ be an odd cactus, $C$ be a cycle in $G$, edge $u_1u_2\in{E(C)}$ such that $w:=\opp(u_1u_2)\in\cutvx{}$, and let $(G_1,G_2)$ be the separation of $G$ with respect to $(u_1u_2,w)$. 
For any block $B$ in $G$,
if the vertex set of the (unique) $B,C$ path in $\bct{G}$ does not contain $w$, then $E(B)\cap{E(G_2)}=\emptyset$.
\end{lemma}

\begin{proof}
We show the contrapositive---that is, we show that if there exists $u_3u_4\in{E(B)}\cap{E(G_2)}$ then $w\in{V(P)}$, where $P$ is the (unique) $B,C$ path in $\bct{G}$. Because $G_2$ is connected, there exists a path from $w$ to $u_3$ in $G_2$. There are two cases: either $w$ is a cut vertex in $G_2$, or it is not a cut vertex in $G_2$. If $w$ is \textit{not} a cut vertex in $G_2$, then there exists a unique block $B'\in\mathcal{B}(G_2)$ such that $w\in{V(B')}$, and a $B',B$ path $P'$ contained in $\bct{G_2}$. Moreover, because each block in $\mathcal{B}(G_2)$ is a block in $\mathcal{B}(G)$ and each cut vertex in $G_2$ is a cut vertex in $G$, the path $P'$ is also a path in $\bct{G}$. By concatenating $P'$ with the edges $(C,w)$ and $(w,B')$ in $\bct{G}$, we obtain a $B,C$ path in $\bct{G}$ containing $w$, as was to be shown.

Otherwise, if $w$ \textit{is} a cut vertex in $G_2$,  
then there exists a $w,B$ path $P'$ in $\bct{G_2}$. By the same argument as above, $P'$ is also a path in $\bct{G}$. By concatenating the path $P'$ with the edge $(C,w)$ in $\bct{G}$, we obtain a $B,C$ path in $\bct{G}$ containing $w$, as was to be shown.
\end{proof}

\begin{theorem}\label{thm:src}
Let $G$ be an odd cactus that is not a cycle. 
Then the function $c$ returned by \autoref{alg:1}
is a $k$-coloring of $E(G)$ that strongly rainbow connects $G$.
Moreover, $k=src(G)$.
\end{theorem}
\begin{proof}
Let $\parti{}$ be the black white partition \refp{eqn:bw-all}, let $k = |E_B|$ and let $c$ be the function returned by \autoref{alg:1} with input graph $G$. Because \autoref{cor:bwbound} ensures that $k \leq src(G)$, it suffices to show that $c$ is a $k$-coloring of $E(G)$ and that $c$ strongly rainbow connects $G$.

We first show that the function $c$ is a $k$-coloring of $E(G)$. Since $G$ is a cactus graph, any edge $e$ in $E(G)$ must either be a cut edge or contained in the edge set of exactly one cycle. If edge $e$ is a cut edge then $c(e)$ is assigned in line \ref{line:4}. If edge $e$ is in the edge set of a cycle $C$, then $e$ is either in a cycle segment in $G$ or $e \in \eant$. If edge $e$ is contained in a cycle segment in $\mathcal{S}_1$ ($\mathcal{S}_2$), then $c(e)$ is assigned in line \ref{line:8}  (line \ref{line:14}). If edge $e$ is contained in a cycle segment in $\mathcal{S}_3$ ($\mathcal{S}_4$), then by \autoref{lem:opp-seg}, edge $e$ must be antipodal to a vertex in a cycle segment in $\mathcal{S}_2$ ($\mathcal{S}_1$) and thus $c(e)$ is assigned in line \ref{line:15} (line \ref{line:10}). If instead edge $e \in \eant$, then $\opp(e)\in\cutvx{G}$. Let $(G_1,G_2)$ denote the separation of $G$ with respect to the antipodal pair $(e,\opp(e))$, so that $G_1$ is the component of $G-\opp(e)$ containing $e$, and $G_2=G[V(G)\setminus{V(G_1)}]$. Then there exists a block $B$ in $G$ such that $B$ is contained in $G_2$, and $B$ is the label of a leaf node in $\bct{G}$. By \autoref{lem:bct}, $E(B)\cap{E_B}\neq\varnothing$. Clearly $E(B)\subseteq{E(G_2)}$, thus $E_B\cap{E(G_2)}\neq\varnothing$, and the edge $e'$ arbitrarily selected in line \ref{line:18} is well defined. The value $c(e)$ is thus assigned in line \ref{line:19}. The domain of function $c$ is $E(G)$.

Throughout \autoref{alg:1}, whenever a value $c(e)$ is assigned for an edge $e \in E(G)$, $c(e)$ is either initialized with the value $color$ or $c(e')$ where $e'$ is an edge in $E_B$. Additionally, $color$ is initialized with value $0$ in line \ref{line:1} and is incremented exactly once for each edge in $E_B$. Since $G$ is assumed to be an odd cactus distinct from a cycle, $E_B$ cannot be empty and thus $color$ must be incremented before the first $c(e)$ is assigned. Thus the range of $c$ is the set $\{1, \dots, k\}$. Thus, function $c$ is a $k$-coloring of $E(G)$.

We complete our proof by showing that the edge coloring $c$ strongly rainbow connects $G$. To do this, we consider an arbitrary pair of distinct vertices $v_a, v_b$ in $V(G)$ and prove that the unique shortest $v_a, v_b$ path $P$ in $G$ must be a rainbow with respect to the edge coloring $c$. If $P$ has length 1 then clearly $P$ is rainbow, so assume that $P$ has at least two distinct edges and consider an arbitrary pair of such edges $u_1 u_2$ and $u_3 u_4$. We assume for contradiction that $c(u_1 u_2) = c(u_3 u_4)$. It is clear from \autoref{alg:1} that if $c(u_1 u_2) = c(u_3 u_4)$ then, up to symmetry, either $u_1 u_2$ is an edge in a segment $S \in \mathcal{S}_1 \cup \mathcal{S}_2$ of a cycle $C$ and $u_3 u_4 = \opp(u_2, C)$, or at least one of $u_1 u_2, u_3 u_4$ are in $\eant$. By deriving a contradiction in either case, we verify that $P$ is indeed rainbow with respect to $c$ and thus that every pair of vertices in $V(G)$ are rainbow connected. 

First, consider the case that $u_1 u_2$ is an edge in a segment $S \in \mathcal{S}_1 \cup \mathcal{S}_2$ of a cycle $C$ and $u_3 u_4 = \opp(u_2, C)$. By \autoref{lem:l0}, since the vertex $u_2 \in V(C)$ and its antipodal edge $\opp(u_2, C) = u_3 u_4 \in E(C)$, $P$ cannot be a shortest $v_a v_b$ path in $G$, a contradiction. Next, we consider the case that exactly one of $u_1 u_2, u_3 u_4$ are in $\eant$. Assume without loss of generality that $u_1 u_2 \in \eant$ and $u_3 u_4 \not \in \eant$. By lines \ref{line:18} and \ref{line:19}, edge $u_3 u_4$ must be in both $E_B$ and $E(G_2)$ where $(G_1, G_2)$ is the separation of $G$ with respect to the antipodal pair $(u_1 u_2, \opp(u_1 u_2))$. Additionally, let $B$ be the block in $G$ such that $u_3 u_4 \in E(B)$. If $V(P)$ contains $\opp(u_1 u_2)$, then by \autoref{lem:l0}, $P$ cannot be the shortest $v_a, v_b$ path in $G$. If $\opp(u_1 u_2) \not \in V(P)$, then since $P$ is a path in $G$ such that $u_1 u_2, u_3 u_4$ are both in $E(P)$ and the cut vertex $\opp(u_1 u_2) \not \in V(G)$, the block-cut tree $\bct{G}$ must contain a path $P'$ between vertices $B, C \in V(\bct{G})$ such that $V(P')$. By \autoref{lem:opp-comp} then, $E(B) \cap E(G_2) = \emptyset$. This is a contradiction as $u_3 u_4$ was chosen to be in $E(G_2)$ and $B$ is simply the block in $G$ such that $u_3 u_4 \in E(B)$. 

Finally, we consider the case that both $u_1 u_2, u_3 u_4$ are in $\eant$. It is clear that \autoref{alg:1} assigns distinct values of $c(e)$ for each edge $e \in E_B$. Thus there exists a unique edge $e' \in E_B$ such that $c(e') = c(u_1 u_2) = c(u_3 u_4)$. Then edge $e'$ must be the edge chosen in line \ref{line:19} when $c(u_1 u_2), c(u_3 u_4)$ are each assigned. By lines \ref{line:17} and \ref{line:18} then, edge $e'$ must be in $E(G_2) \cap E(G_4)$, where $(G_1, G_2)$ and $(G_3, G_4)$ are the separations of $G$ for the antipodal pairs $(u_1 u_2, \opp(u_1 u_2))$ and $(u_3 u_4, \opp(u_3 u_4))$, respectively. Let $B_1, B_2$, and $B'$ be the blocks whose edge sets contain $u_1 u_2, u_3 u_4$, and $e'$, respectively. Since edge $e' \in E(G_2) \cap E(G_4)$, $E(B_1) \subseteq E(G_1)$, and $E(B_2) \subseteq E(G_3)$, block $B'$ must be distinct from both $B_1$ and $B_2$. If block $B = B_1 = B_2$, then since no two edges in a cycle can be antipodal to the same cut vertex, it follows that $\opp(u_1 u_2) \neq \opp(u_3 u_4)$. The block-cut tree $\bct{G}$ must then contain distinct vertices labeled $B$, $\opp(u_1 u_2)$, $\opp(u_3 u_4)$, and in $\bct{G}$ the vertices $\opp(u_1 u_2)$, $\opp(u_3 u_4)$ must each be adjacent to vertex $B$. Since $\bct{G}$ is a tree, the graph $G' = \bct{G} - B$ must then contain distinct components $K_1, K_2$ such that $\opp(u_1 u_2) \in V(K_1)$ and $\opp(u_3 u_4) \in V(K_2)$. Then $E(G_2), E(G_4)$ are the sets of edges in $E(G)$ which are contained in the vertex set of a block in $V(K_1), V(K_2)$, respectively, and thus $E_2$ and $E_4$ must be disjoint. This contradicts the previous assertion that edge $e'$ must be in $E(G_2) \cap E(G_4)$, so we conclude that $B_1 \neq B_2$.

Next, consider the $B_1, B'$ path $P'_{bct}$ in the block-cut tree $\bct{G}$. Assume without loss of generality, that vertex $B_2$ is not in $V(P'_{bct})$ (if vertex $B_2$ is contained in this path then the $B_1, B$ path in $\bct{G}$ does not contain $B_1$ and $V(G)$ can be relabelled). If $V(P)$ contains $\opp(u_3 u_4)$ then by \autoref{lem:l0}, $P$ cannot be a shortest $v_a, v_b$ path in $G$, so assume that $V(P)$ does not contain vertex $\opp(u_3 u_4)$. The existence of $P$ implies the existence of a $B_1, B_2$ path $P_{bct}$ in $\bct{G}$ which does not include the vertex $\opp(u_3 u_4)$. Next, either $V(P'_{bct})$ contains $\opp(u_3 u_4)$ or $V(P'_{bct})$ does not contain $\opp(u_3 u_4)$. If $V(P'_{bct})$ contains $\opp(u_3 u_4)$, then since $\opp(u_3 u_4) \not \in V(P_bct)$ and $C_2 \not \in V(P'_bct)$, $\bct{G}$ has a $B_2, \opp(u_3 u_4)$ path which does not contain the edge $B_2 \opp(u_3 u_4)$. Since $B_2 \opp(u_3 u_4) \in E(\bct{G})$ however, the block tree $\bct{G}$ contains a cycle, which is a contradiction. If instead $V(P'_{bct})$ does not contain the vertex $\opp(u_3 u_4)$, then $\bct{G}$ contains a $B_2, B'$ path whose vertex set does not contain $\opp(u_3 u_4)$. By \autoref{lem:opp-comp} then, $e'$ cannot have been chosen in line \ref{line:18}, a contradiction.

We have shown that in each case, the assumption $c(u_1 u_2) = c(u_3 u_4)$ leads to a contradiction. We conclude that $c(u_1 u_2) \neq c(u_3 u_4)$, and thus that if $P$ is a shortest $v_a, v_b$ path in $G$ then $P$ must be rainbow with respect to the edge coloring $c$. Since vertices $v_a, v_b$ were chosen arbitrarily in $V(G)$, we have shown that the edge coloring $c$ strongly rainbow connects $G$.
\end{proof}

\begin{corollary}\label{cor:formula}
For any odd cactus $G$, formula \refp{eqn:main-formula} gives the value of the strong rainbow connection number $src(G)$ of $G$.
\end{corollary}
\begin{proof}
The formula for the case of $G$ being a(n odd) cycle was established by \citet{chartrand2008}. Now, suppose $G$ is not an odd cycle.
From the proof of \autoref{thm:src}, we have that the strong rainbow coloring produced by \autoref{alg:1} uses $k=|E_B|$ colors, where $E_B$ is given by \refp{eqn:bw-eb}. Moreover, $k=src(G)$. Hence, we have that $src(G)=|\cutes|+\sum_{S\in\mathcal{S}_1\cup\mathcal{S}_2}|E(S)|$.

Next, from \autoref{lem:opp-seg} we observe that, for any cycle segment $S\in\mathcal{S}_1$ and its opposite segment $S'\in\mathcal{S}_4$, $|E(S)|=|E(S')|+1$. Similarly, for any cycle segment $S\in\mathcal{S}_2$ and its opposite segment $S'\in\mathcal{S}_3$, $|E(S)|=|E(S')|$. Summing these identities over all the cycle segments in $G$, we have that 
\begin{equation}\label{eqn:cor-step}\tag{$\star$}
    |\mathcal{S}_1|+\sum_{S\in\mathcal{S}_3\cup\mathcal{S}_4}|E(S)|=\sum_{S\in\mathcal{S}_1\cup\mathcal{S}_2}|E(S)|\quad\Rightarrow\quad
    |\mathcal{S}_1|+\sum_{S\in\mathcal{S}}|E(S)|=2\sum_{S\in\mathcal{S}_1\cup\mathcal{S}_2}|E(S)|,
\end{equation}
where, for notational brevity, we define $\mathcal{S}:=\mathcal{S}_1\cup\mathcal{S}_2\cup\mathcal{S}_3\cup\mathcal{S}_4$.
Now, because $\eant$, $\cutes$ and the edge sets of the cycle segments in $G$ form a partition of $E(G)$, we have that 
\[
    m=|\eant|+|\cutes|+\sum_{S\in\mathcal{S}}|E(S)|
    =|\eant|+|\cutes|+2\sum_{S\in\mathcal{S}_1\cup\mathcal{S}_2}|E(S)|-|\mathcal{S}_1|,
\]
where the second equality follows from \refp{eqn:cor-step}. Rearranging this identity and plugging into the above formula for $src(G)$, we obtain
\[
    src(G)=|\cutes|+\sum_{S\in\mathcal{S}_1\cup\mathcal{S}_2}|E(S)|
    =|\cutes|+\tfrac{1}{2}\big(m+|\mathcal{S}_1|-|\eant|-|\cutes|\big)
    =\tfrac{1}{2}\big(m+|\cutes|+|\mathcal{S}_1|-|\eant|\big),
\]
as was to be shown.
\end{proof}

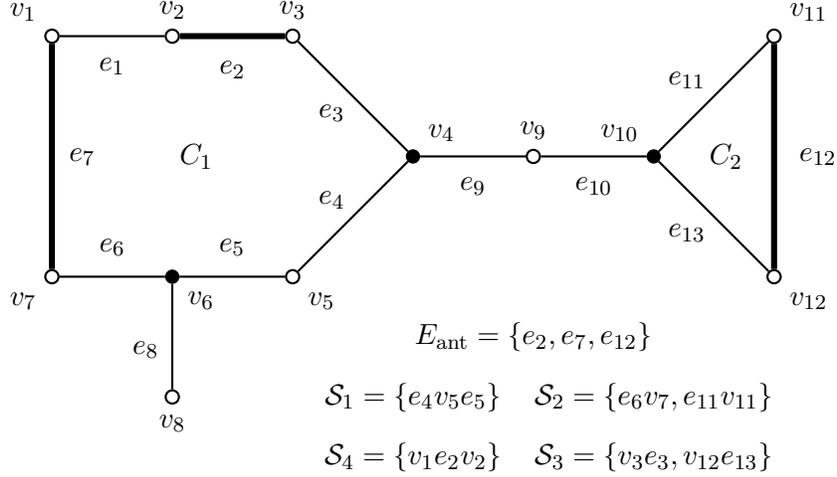
\begin{figure}
\centering
\input{example02.tex}
\caption{Example graph to illustrate \autoref{alg:1}. Dark vertices are in $\cutvx{}$, and dark edges are in $\eant$. For this graph, $src(G)=7$ (cf.~formula \refp{eqn:main-formula}).}
\label{fig:example02}
\end{figure}

We illustrate our results with the example shown in \autoref{fig:example02}. For this odd cactus, we have $m=13$, $\cutes=\{e_8,e_9,e_{10}\}$, $\mathcal{S}_1=\{e_4v_5e_5\}$ and $\eant=\{e_2,e_7,e_{12}\}$, and thus, applying formula \refp{eqn:main-formula}, $src(G)=\tfrac{1}{2}(13+3+1-3)=7$. \autoref{alg:1} produces a strong rainbow coloring using $7$ colors in the following steps:
\begin{enumerate}
\item First, the for loop in line \ref{line:2} colors each edge in $\cutes$ a different color. Say, $c(e_8)=1$, $c(e_9)=2$ and $c(e_{10})=3$. At the termination of the for loop in line \ref{line:2}, $\mathrm{color}=3$.
\item Next, the for loop in line \ref{line:5} iterates over the segments in $\mathcal{S}_1$---in this example, the single segment $e_4v_4e_5$ contained in cycle $C_1$. The inner for loop at line \ref{line:6} first colors $c(e_4)=4$, and the condition in line \ref{line:9} is satisfied, so that $c(e_1)=4$ as well. The next iteration of the inner loop colors $c(e_5)=5$, and the conditional at line \ref{line:9} is skipped. At the termination of the for loop in line \ref{line:5}, $\mathrm{color}=5$. Note that, at the termination of the for loop at line \ref{line:5}, all of the edges contained in both $\mathcal{S}_1$ cycle segments and $\mathcal{S}_4$ segments have been colored.
\item Next, the for loop in \ref{line:11} iterates over the segments in $\mathcal{S}_2$. Suppose that the algorithm first considers segment $e_6v_7$ in cycle $C_1$. Then, per lines \ref{line:14} and \ref{line:15}, the algorithm sets $c(e_6)=c(e_3)=6$. Next, the algorithm considers segment $e_{11}v_{11}$ in cycle $C_2$, and colors $c(e_{11})=c(e_{13})=7$. At the termination of the for loop in line \ref{line:11}, the counter $\mathrm{color}=7$, equal to $src(G)$. Note that, at the termination of the for loop at line \ref{line:11}, all of the edges contained in both $\mathcal{S}_2$ and $\mathcal{S}_3$ segments have been colored.
\item Finally, the edges in $\eant$ are colored in the for loop at line \ref{line:16} (all other edges have been colored). The edge $e_2$ may only be colored $c(e_2)=1$, the same color as edge $e_8$ (because $e_8$ is the only edge satisfying the conditions in line \ref{line:18}). The color for edge $e_7$ may be any of $\{c(e_9),c(e_{10}),c(e_{11})\}$, and the color for edge $e_{12}$ may be any of $\{c(e_4),c(e_5),c(e_6),c(e_9),c(e_{10})\}$.
\end{enumerate}


\section{Conclusion and Future Work}

In this paper we present both a formula \refp{eqn:main-formula} and a polynomial time complexity algorithm (\autoref{alg:1}) for computing the strong rainbow connection numbers of a wide class of cactus graphs, those that do not contain even length cycles. To achieve this, we first introduce the notion of black white partitions defined for odd cactus graphs, which partition the vertices and edges of an odd cactus. One of the edge sets in this partition scheme has the property that, for each pair of edges in the set, there exists a pair of vertices such that the unique shortest path between them traverses both edges. Since any strong rainbow edge coloring of the graph must then map these edges to distinct colors, this partitioning scheme provides a lower bound for the strong rainbow connection number of the graph (as shown by \autoref{thm:bwcoloring}). We then show that the structure of the cut vertices and blocks in odd cacti can be used to directly construct a special black white partition (as shown by \autoref{thm:bw}). In particular, the strong rainbow coloring constructed by \autoref{alg:1} (shown to be correct by \autoref{thm:src}) achieves the bound that this special black white partition presents. Thus, \autoref{alg:1} produces an optimal strong rainbow coloring. Moreover, the number of colors used in the edge coloring produced by \autoref{alg:1} is directly given by \refp{eqn:main-formula}.

Despite fairly wide interest in both rainbow connection and strong rainbow connection numbers of graphs, relatively few polynomial algorithms have been presented for computing strong rainbow connection numbers. In this paper, we show that the strong rainbow connection numbers of many cacti can be computed in worst case polynomial time complexity. The complexity of computing the strong rainbow connection numbers in general cacti however remains an open question. The addition of even length cycles adds a significant degree of difficulty to the problem in the case of cacti. While shortest paths between vertices in odd cacti are unique, the number of distinct shortest paths between vertices in general cacti may be asymptotically exponential. Nevertheless, we conclude our study with the following conjecture.

\begin{conjecture}
There exists an algorithm which computes $src(G)$ for any cactus graph $G$ with worst case polynomial time complexity. 
\end{conjecture}


\section*{Acknowledgements}
This work was supported by National Science Foundation grant number DMS-1720225. Additionally, D. Mildebrath was supported by the United States Department of Defense through the National Defense Science and Engineering Graduate Fellowship program.

\bibliography{bib}

\end{document}

%% file: segments.tex
\begin{tikzpicture}[x=2cm,y=2cm,scale=0.8]

\tikzstyle{whitenode} = [draw,circle,minimum size=5pt,inner sep=0pt, thick]
\tikzstyle{blacknode} = [draw,circle,minimum size=5pt,inner sep=0pt, fill]
\tikzstyle{whiteedge} = [line width=0.8mm]
\tikzstyle{blackedge} = [thick]
\tikzstyle{twist}     = [label distance=-1cm, text depth=0ex, rotate=-90]


\pgfmathsetmacro{\xa}{ 1.0} \pgfmathsetmacro{\ya}{ 1.0}
\pgfmathsetmacro{\xb}{ 1.0} \pgfmathsetmacro{\yb}{ 0.0}
\pgfmathsetmacro{\xc}{ 1.0} \pgfmathsetmacro{\yc}{-1.0}
\pgfmathsetmacro{\xd}{ 0.0} \pgfmathsetmacro{\yd}{-2.0}
\pgfmathsetmacro{\xe}{-1.0} \pgfmathsetmacro{\ye}{-2.0}
\pgfmathsetmacro{\xf}{-2.0} \pgfmathsetmacro{\yf}{-1.0}
\pgfmathsetmacro{\xg}{-2.0} \pgfmathsetmacro{\yg}{ 0.0}
\pgfmathsetmacro{\xh}{-1.0} \pgfmathsetmacro{\yh}{ 1.0}
\pgfmathsetmacro{\xi}{ 0.0} \pgfmathsetmacro{\yi}{ 1.0}
\pgfmathsetmacro{\xj}{ 2.0} \pgfmathsetmacro{\yj}{ 1.5}
\pgfmathsetmacro{\xk}{ 1.4} \pgfmathsetmacro{\yk}{-2.0}

\node[blacknode,label=above:$v_1$] (a) at (\xa,\ya) {};
\node[whitenode,label=right:$v_2$] (b) at (\xb,\yb) {};
\node[blacknode,label=right:$v_3$] (c) at (\xc,\yc) {};
\node[whitenode,label=below:$v_4$] (d) at (\xd,\yd) {};
\node[whitenode,label=below:$v_5$] (e) at (\xe,\ye) {};
\node[whitenode,label=left:$v_6$] (f) at (\xf,\yf) {};
\node[whitenode,label=left:$v_7$] (g) at (\xg,\yg) {};
\node[whitenode,label=above:$v_8$] (h) at (\xh,\yh) {};
\node[whitenode,label=above:$v_9$] (i) at (\xi,\yi) {};
\node[whitenode] (j) at (\xj,\yj) {};
\node[whitenode] (k) at (\xk,\yk) {};

\draw[blackedge] (a) -- (b) node [midway, yshift=-5pt, label=left:$e_1$] {};
\draw[blackedge] (b) -- (c) node [midway, label=left:$e_2$] {};
\draw[blackedge] (c) -- (d) node [midway, xshift=5pt, yshift=-5pt, label=above left:$e_3$] {};
\draw[blackedge] (d) -- (e) node [midway, label=above:$e_4$] {};
\draw[whiteedge] (e) -- (f) node [midway, xshift=-5pt, yshift=-5pt, label=above right:$e_5$] {};
\draw[blackedge] (f) -- (g) node [midway, label=right:$e_6$] {};
\draw[whiteedge] (g) -- (h) node [midway, xshift=-5pt, yshift=5pt, label=below right:$e_7$] {};
\draw[blackedge] (h) -- (i) node [midway, label=below:$e_8$] {};
\draw[blackedge] (i) -- (a) node [midway, xshift=-5pt, label=below:$e_9$] {};
\draw[blackedge] (a) -- (j);
\draw[blackedge] (c) -- (k);

\pgfmathsetmacro{\xx}{-4}
\pgfmathsetmacro{\yy}{-3}
\pgfmathsetmacro{\gap}{0.4}
\pgfmathsetmacro{\bump}{0.3}

\node[draw] (a00) at (\xx+ 0*\gap,\yy) {$v_1$};
\node[]     (a01) at (\xx+ 1*\gap,\yy) {$e_1$};
\node[]     (a02) at (\xx+ 2*\gap,\yy) {$v_2$};
\node[]     (a03) at (\xx+ 3*\gap,\yy) {$e_2$};
\node[draw] (a04) at (\xx+ 4*\gap,\yy) {$v_3$};
\node[]     (a05) at (\xx+ 5*\gap,\yy) {$e_3$};
\node[]     (a06) at (\xx+ 6*\gap,\yy) {$v_4$};
\node[]     (a07) at (\xx+ 7*\gap,\yy) {$e_4$};
\node[]     (a08) at (\xx+ 8*\gap,\yy) {$v_5$};
\node[draw] (a09) at (\xx+ 9*\gap,\yy) {$e_5$};
\node[]     (a10) at (\xx+10*\gap,\yy) {$v_6$};
\node[]     (a11) at (\xx+11*\gap,\yy) {$e_6$};
\node[]     (a12) at (\xx+12*\gap,\yy) {$v_7$};
\node[draw] (a13) at (\xx+13*\gap,\yy) {$e_7$};
\node[]     (a14) at (\xx+14*\gap,\yy) {$v_8$};
\node[]     (a15) at (\xx+15*\gap,\yy) {$e_8$};
\node[]     (a16) at (\xx+16*\gap,\yy) {$v_9$};
\node[]     (a17) at (\xx+17*\gap,\yy) {$e_9$};
\node[draw] (a18) at (\xx+18*\gap,\yy) {$v_1$};

\draw[decorate,decoration={brace,mirror,amplitude=4pt}]
    (\xx+\gap,\yy-\bump) -- (\xx+3*\gap,\yy-\bump) node
    [midway,below,yshift=-5pt] {$\mathcal{S}_1$ type};

\draw[decorate,decoration={brace,mirror,amplitude=4pt}]
    (\xx+5*\gap,\yy-\bump) -- (\xx+8*\gap,\yy-\bump) node
    [midway,below,yshift=-5pt] {$\mathcal{S}_2$ type};

\draw[decorate,decoration={brace,mirror,amplitude=4pt}]
    (\xx+10*\gap,\yy-\bump) -- (\xx+12*\gap,\yy-\bump) node
    [midway,below,yshift=-5pt] {$\mathcal{S}_4$ type};

\draw[decorate,decoration={brace,mirror,amplitude=4pt}]
    (\xx+14*\gap,\yy-\bump) -- (\xx+17*\gap,\yy-\bump) node
    [midway,below,yshift=-5pt] {$\mathcal{S}_3$ type};

\end{tikzpicture}

%% file: separation.tex
\begin{tikzpicture}

\pgfmathsetmacro{\shift}{4}
\pgfmathsetmacro{\gap}{-0.5}

\tikzstyle{whitenode} = [draw,circle,minimum size=5pt,inner sep=0pt, thick]
\tikzstyle{blacknode} = [draw,circle,minimum size=5pt,inner sep=0pt, fill]
\tikzstyle{whiteedge} = [line width=0.8mm]
\tikzstyle{blackedge} = [thick]


\pgfmathsetmacro{\xa}{4.60} \pgfmathsetmacro{\ya}{ 0.50}
\pgfmathsetmacro{\xb}{3.80} \pgfmathsetmacro{\yb}{ 1.50}
\pgfmathsetmacro{\xc}{3.80} \pgfmathsetmacro{\yc}{ 0.70}
\pgfmathsetmacro{\xd}{3.50} \pgfmathsetmacro{\yd}{-1.00}
\pgfmathsetmacro{\xe}{3.00} \pgfmathsetmacro{\ye}{ 0.00}
\pgfmathsetmacro{\xf}{2.40} \pgfmathsetmacro{\yf}{ 0.66}
\pgfmathsetmacro{\xg}{1.50} \pgfmathsetmacro{\yg}{ 1.20}
\pgfmathsetmacro{\xh}{2.40} \pgfmathsetmacro{\yh}{-0.50}
\pgfmathsetmacro{\xi}{2.70} \pgfmathsetmacro{\yi}{-1.40}
\pgfmathsetmacro{\xj}{1.60} \pgfmathsetmacro{\yj}{-0.70}

\node[whitenode] (a) at (\xa,\ya) {};
\node[whitenode] (b) at (\xb,\yb) {};
\node[whitenode] (c) at (\xc,\yc) {};
\node[whitenode] (d) at (\xd,\yd) {};
\node[blacknode,label=right:$v$] (e) at (\xe,\ye) {};
\node[whitenode] (f) at (\xf,\yf) {};
\node[whitenode] (g) at (\xg,\yg) {};
\node[whitenode] (h) at (\xh,\yh) {};
\node[whitenode] (i) at (\xi,\yi) {};
\node[whitenode] (j) at (\xj,\yj) {};

\node[whitenode] (a0) at (\xa+\shift,\ya) {};
\node[whitenode] (b0) at (\xb+\shift,\yb) {};
\node[whitenode] (c0) at (\xc+\shift,\yc) {};
\node[whitenode] (d0) at (\xd+\shift,\yd) {};
\node[blacknode] (e1) at (\xe+\shift,\ye) {}; 

\node[blacknode] (e2) at (\xe+\shift+\gap,\ye) {};
\node[whitenode] (f0) at (\xf+\shift+\gap,\yf) {};
\node[whitenode] (g0) at (\xg+\shift+\gap,\yg) {};
\node[whitenode] (h0) at (\xh+\shift+\gap,\yh) {};
\node[whitenode] (i0) at (\xi+\shift+\gap,\yi) {};
\node[whitenode] (j0) at (\xj+\shift+\gap,\yj) {};

\draw[blackedge] (a) -- (b); \draw[blackedge] (a0) -- (b0);
\draw[blackedge] (a) -- (c); \draw[blackedge] (a0) -- (c0);
\draw[blackedge] (b) -- (c); \draw[blackedge] (b0) -- (c0);
\draw[blackedge] (c) -- (e); \draw[blackedge] (c0) -- (e1);
\draw[blackedge] (d) -- (e); \draw[blackedge] (d0) -- (e1);

\draw[thick, dotted] (e1) -- (e2);

\draw[blackedge] (e) -- (f); \draw[blackedge] (e2) -- (f0);
\draw[blackedge] (e) -- (h); \draw[blackedge] (e2) -- (h0);
\draw[blackedge] (f) -- (g); \draw[blackedge] (f0) -- (g0);

\draw[whiteedge] (f) -- (h) node [midway, left] {$e$};
\draw[whiteedge] (f0) -- (h0);

\draw[blackedge] (h) -- (i); \draw[blackedge] (h0) -- (i0);
\draw[blackedge] (h) -- (j); \draw[blackedge] (h0) -- (j0);
\draw[blackedge] (i) -- (j); \draw[blackedge] (i0) -- (j0);

\draw[decorate,decoration={brace,mirror,amplitude=10pt}]
(4.8,-1.6) -- (6.4,-1.6) node [midway,below, yshift=-10pt]
{\footnotesize $G_1$};

\draw[decorate,decoration={brace,mirror,amplitude=10pt}]
(6.8,-1.6) -- (8.6,-1.6) node [midway,below, yshift=-10pt]
{\footnotesize $G_2$};

\end{tikzpicture}

%% file: paths.tex
\begin{tikzpicture}

\pgfmathsetmacro{\alph}{0.5} 

\newcommand{\dasharc}[5]{ 
    \draw[thick] (#1,#2) -- (#1-0.5*#5*#1+0.5*#5*#3,
                        #2-0.5*#5*#2+0.5*#5*#4);
    \draw[thick] (#3,#4) -- (#3-0.5*#5*#3+0.5*#5*#1,
                        #4-0.5*#5*#4+0.5*#5*#2);
    \draw[dotted,very thick] 
          (#1-0.5*#5*#1+0.5*#5*#3,#2-0.5*#5*#2+0.5*#5*#4) --
          (#3-0.5*#5*#3+0.5*#5*#1,#4-0.5*#5*#4+0.5*#5*#2);
}

\tikzstyle{mynode} = [draw, fill, circle, minimum size=5pt, inner sep=0pt]


\pgfmathsetmacro{\xa}{1} \pgfmathsetmacro{\ya}{ 0}
\pgfmathsetmacro{\xb}{2} \pgfmathsetmacro{\yb}{ 0}
\pgfmathsetmacro{\xc}{3} \pgfmathsetmacro{\yc}{ 2}
\pgfmathsetmacro{\xd}{4} \pgfmathsetmacro{\yd}{ 2}
\pgfmathsetmacro{\xe}{5} \pgfmathsetmacro{\ye}{ 2}
\pgfmathsetmacro{\xf}{6} \pgfmathsetmacro{\yf}{ 2}
\pgfmathsetmacro{\xg}{3} \pgfmathsetmacro{\yg}{-2}
\pgfmathsetmacro{\xh}{6} \pgfmathsetmacro{\yh}{-2}
\pgfmathsetmacro{\xi}{7} \pgfmathsetmacro{\yi}{ 0}
\pgfmathsetmacro{\xj}{8} \pgfmathsetmacro{\yj}{ 0}

\node[mynode,label={above:$a=v_1$}]       (a) at (\xa,\ya) {};
\node[mynode,label={right:$v_{i_1}=w_1$}] (b) at (\xb,\yb) {};
\node[mynode,label={above:$w_2$}]         (c) at (\xc,\yc) {};
\node[mynode,label={above:$w_{k_1}$}]     (d) at (\xd,\yd) {};
\node[mynode,label={above:$w_{k_1+1}$}]   (e) at (\xe,\ye) {};
\node[mynode,label={above:$w_{k_2}$}]     (f) at (\xf,\yf) {};
\node[mynode,label={below:$w_r$}]         (g) at (\xg,\yg) {};
\node[mynode,label={below:$w_{q+1}$}]     (h) at (\xh,\yh) {};
\node[mynode,label={left:$w_q=v_{i_2}$}]  (i) at (\xi,\yi) {};
\node[mynode,label={above:$v_\ell=b$}]    (j) at (\xj,\yj) {};

\dasharc{\xa}{\ya}{\xb}{\yb}{0.5}
\dasharc{\xc}{\yc}{\xd}{\yd}{0.4}
\dasharc{\xe}{\ye}{\xf}{\yf}{0.4}
\dasharc{\xf}{\yf}{\xi}{\yi}{\alph}
\dasharc{\xi}{\yi}{\xj}{\yj}{0.5}
\dasharc{\xg}{\yg}{\xh}{\yh}{\alph}
\dasharc{\xb}{\yb}{\xc}{\yc}{1} 
\dasharc{\xd}{\yd}{\xe}{\ye}{1}
\dasharc{\xb}{\yb}{\xg}{\yg}{1}
\dasharc{\xh}{\yh}{\xi}{\yi}{1}

\draw[thick,->] (3,1)  .. controls (4,2)  and (5,2)  .. (6,1);
\draw[thick,->] (3,-1) .. controls (4,-2) and (5,-2) .. (6,-1);
\node at (4.5,1.3) {$P_1$};
\node at (4.5,-1.3) {$P_2$};

\end{tikzpicture}

%% file: example02.tex
\begin{tikzpicture}[x=2cm,y=2cm,scale=0.8]

\usetikzlibrary{decorations.pathreplacing}

\tikzstyle{whitenode} = [draw,circle,minimum size=5pt,inner sep=0pt, thick]
\tikzstyle{blacknode} = [draw,circle,minimum size=5pt,inner sep=0pt, fill]
\tikzstyle{whiteedge} = [line width=0.8mm]
\tikzstyle{blackedge} = [thick]
\tikzstyle{twist}     = [label distance=-1cm, text depth=0ex, rotate=-90]


\pgfmathsetmacro{\xa}{0} \pgfmathsetmacro{\ya}{4}
\pgfmathsetmacro{\xb}{1} \pgfmathsetmacro{\yb}{4}
\pgfmathsetmacro{\xc}{2} \pgfmathsetmacro{\yc}{4}
\pgfmathsetmacro{\xd}{3} \pgfmathsetmacro{\yd}{3}
\pgfmathsetmacro{\xe}{2} \pgfmathsetmacro{\ye}{2}
\pgfmathsetmacro{\xf}{1} \pgfmathsetmacro{\yf}{2}
\pgfmathsetmacro{\xg}{0} \pgfmathsetmacro{\yg}{2}
\pgfmathsetmacro{\xh}{1} \pgfmathsetmacro{\yh}{1}
\pgfmathsetmacro{\xi}{4} \pgfmathsetmacro{\yi}{3}
\pgfmathsetmacro{\xj}{5} \pgfmathsetmacro{\yj}{3}
\pgfmathsetmacro{\xk}{6} \pgfmathsetmacro{\yk}{4}
\pgfmathsetmacro{\xl}{6} \pgfmathsetmacro{\yl}{2}


\node[whitenode,label=above left:$v_1$]     (a) at (\xa,\ya) {};
\node[whitenode,label=above:$v_2$]          (b) at (\xb,\yb) {};
\node[whitenode,label=above:$v_3$]          (c) at (\xc,\yc) {};
\node[blacknode,label=above right:$v_4$]    (d) at (\xd,\yd) {};
\node[whitenode,label=below right:$v_5$]    (e) at (\xe,\ye) {};
\node[blacknode,label=below right:$v_6$]    (f) at (\xf,\yf) {};
\node[whitenode,label=below left:$v_7$]     (g) at (\xg,\yg) {};
\node[whitenode,label=below:$v_8$]          (h) at (\xh,\yh) {};
\node[whitenode,label=above:$v_9$]          (i) at (\xi,\yi) {};
\node[blacknode,label=above left:$v_{10}$]  (j) at (\xj,\yj) {};
\node[whitenode,label=above right:$v_{11}$] (k) at (\xk,\yk) {};
\node[whitenode,label=below right:$v_{12}$] (l) at (\xl,\yl) {};

\draw[blackedge] (a) -- (b) node [midway, label=below:$e_1$] {};
\draw[whiteedge] (b) -- (c) node [midway, label=below:$e_2$] {};
\draw[blackedge] (c) -- (d) node [midway, xshift=5pt, yshift=5pt, label=below left:$e_3$] {};
\draw[blackedge] (d) -- (e) node [midway, xshift=5pt, yshift=-5pt, label=above left:$e_4$] {};
\draw[blackedge] (e) -- (f) node [midway, label=above:$e_5$] {};
\draw[blackedge] (f) -- (g) node [midway, label=above:$e_6$] {};
\draw[whiteedge] (g) -- (a) node [midway, xshift=-3pt, label=right:$e_7$] {};
\draw[blackedge] (f) -- (h) node [midway, xshift=3pt, yshift=-5pt, label=left:$e_8$] {};
\draw[blackedge] (d) -- (i) node [midway, label=below:$e_9$] {};
\draw[blackedge] (i) -- (j) node [midway, label=below:$e_{10}$] {};
\draw[blackedge] (j) -- (k) node [midway, xshift=5pt, yshift=-5pt, label=above left:$e_{11}$] {};
\draw[whiteedge] (k) -- (l) node [midway, label=right:$e_{12}$] {};
\draw[blackedge] (l) -- (j) node [midway, xshift=5pt, yshift=5pt, label=below left:$e_{13}$] {};




\node[] at (1.2,3) {$C_1$};
\node[] at (5.6,3) {$C_2$};

\node[] at (3,1) {$\mathcal{S}_1=\{e_4v_5e_5\}$};
\node[] at (3,0.5)   {$\mathcal{S}_4=\{v_1e_2v_2\}$};
\node[] at (5,1) {$\mathcal{S}_2=\{e_6v_7,e_{11}v_{11}\}$};
\node[] at (5,0.5)   {$\mathcal{S}_3=\{v_3e_3,v_{12}e_{13}\}$};

\node[] at (4,1.5) {$\eant=\{e_2,e_7,e_{12}\}$};

\end{tikzpicture}